\documentclass[a4paper,11pt]{article}
\usepackage{a4wide}
\usepackage{amsmath,amssymb,amsfonts}
\usepackage{amsthm}
\usepackage{latexsym}

\usepackage{epsfig}

\usepackage{color}
\usepackage{hyperref} 

\definecolor{modra3}{rgb}{.1,.0,.4}
\hypersetup{colorlinks=true, linkcolor=modra3, urlcolor=modra3, citecolor=modra3, pdfpagemode=UseNone, bookmarks=true, pdfstartview=}

\if10     
\usepackage[mathlines]{lineno}
\newcommand*\patchAmsMathEnvironmentForLineno[1]{%
  \expandafter\let\csname old#1\expandafter\endcsname\csname #1\endcsname
  \expandafter\let\csname oldend#1\expandafter\endcsname\csname end#1\endcsname
  \renewenvironment{#1}%
     {\linenomath\csname old#1\endcsname}%
     {\csname oldend#1\endcsname\endlinenomath}}%
\newcommand*\patchBothAmsMathEnvironmentsForLineno[1]{%
  \patchAmsMathEnvironmentForLineno{#1}%
  \patchAmsMathEnvironmentForLineno{#1*}}%
\AtBeginDocument{%
\patchBothAmsMathEnvironmentsForLineno{equation}%
\patchBothAmsMathEnvironmentsForLineno{align}%
\patchBothAmsMathEnvironmentsForLineno{flalign}%
\patchBothAmsMathEnvironmentsForLineno{alignat}%
\patchBothAmsMathEnvironmentsForLineno{gather}%
\patchBothAmsMathEnvironmentsForLineno{multline}%
}
\linenumbers
\fi

 \newtheorem{theorem}{Theorem}

 \newtheorem{conjecture}[theorem]{Conjecture}
 \newtheorem{problem}[theorem]{Problem}

\begin{document}

\title{The hamburger theorem}

\author{
Mikio Kano\thanks{Supported by JSPS KAKENHI Grant Number 25400187. Ibaraki University, Hitachi, Ibaraki, Japan. E-mail: \texttt{kano@mx.ibaraki.ac.jp}}
  \and 
Jan Kyn\v{c}l\thanks{Supported by Swiss National Science Foundation Grants 200021-137574 and 200020-144531, by the grant no. 14-14179S of the Czech Science Foundation (GA\v{C}R) and by the grant GAUK 1262213 of the Grant Agency of Charles University. Department of Applied Mathematics and Institute for Theoretical Computer Science, Charles University, Faculty of Mathematics and Physics, Malostransk\'e n\'am.~25, 118 00~ Prague, Czech Republic; and \'Ecole Polytechnique F\'ed\'erale de Lausanne, Chair of Combinatorial Geometry, EPFL-SB-MATHGEOM-DCG, Station 8, CH-1015 Lausanne, Switzerland. E-mail: \texttt{kyncl@kam.mff.cuni.cz}}
} 



\date{}

\maketitle

\begin{center}
In memory of Ferran Hurtado and Ji\v{r}\'{i} Matou\v{s}ek
\bigskip
%
\end{center}


\begin{abstract}
We generalize the ham sandwich theorem to $d+1$ measures on $\mathbb{R}^d$ as follows.
Let $\mu_1,\allowbreak \mu_2, \dots,\allowbreak \mu_{d+1}$ be absolutely continuous finite Borel measures on $\mathbb{R}^d$. Let $\omega_i=\mu_i(\mathbb{R}^d)$ for $i\in [d+1]$, $\omega=\min\{\omega_i; i\in [d+1]\}$ and assume that $\sum_{j=1}^{d+1} \omega_j=1$. 
Assume that $\omega_i \le 1/d$
for every $i\in[d+1]$.
Then there exists a hyperplane $h$ 
such that each open halfspace $H$ defined by $h$ satisfies
%
%
$\mu_i(H) \le (\sum_{j=1}^{d+1} \mu_j(H))/d$ for every $i \in [d+1]$ and $\sum_{j=1}^{d+1} \mu_j(H) \ge \min\{1/2, 1-d\omega\} \ge 1/(d+1)$.
As a consequence we obtain that every $(d+1)$-colored set of $nd$ points in $\mathbb{R}^d$ such that no color is used for more than $n$ points can be partitioned into $n$ disjoint rainbow $(d-1)$-dimensional simplices.

\end{abstract}


\section{Introduction}
It is well-known that if $n$ red points and $n$ blue points are
given in the plane in general position,
then there exists a noncrossing perfect matching on these points
where each edge is a straight-line segment and connects a
red point with a blue point. 
Akiyama and Alon~\cite{AA1989} generalized this result to higher dimensions as follows.

For a positive integer $m$, we write $\mathbb{R}^m$ for the $m$-dimensional Euclidean space and $[m]$ for the set $\{1,2,\dots,m\}$.

\begin{theorem}[Akiyama and Alon~\cite{AA1989}]
\label{theorem_1_colorful_partition}
Let $d\ge 2$ and $n\ge 2$ be integers, and for each $i \in [d]$,
let $X_i$ be a set of $n$ points in $\mathbb{R}^d$ such that all $X_i$ are
pairwise disjoint and
no $d+1$ points of $X_1\cup X_2 \cup \dots \cup X_d$ are 
contained in a hyperplane.
Then there exist $n$ pairwise disjoint $(d-1)$-dimensional simplices, 
each of which contains precisely one vertex from each 
$X_i, i\in [d]$. 
\end{theorem}

The planar version of Theorem~\ref{theorem_1_colorful_partition} follows, for example, from the simple fact that a shortest geometric red-blue perfect matching is noncrossing~\cite{AA1989}. This elementary metric argument does not generalize to higher dimensions, however. Akiyama and Alon~\cite{AA1989} proved Theorem~\ref{theorem_1_colorful_partition} using the ham sandwich theorem; see Subsection~\ref{sec_ham_sandwich}.

Aichholzer et al.~\cite{ACFFHHHW10_edge} and Kano, Suzuki and Uno~\cite{KSU2014} extended the planar version of Theorem~\ref{theorem_1_colorful_partition} to an arbitrary number of colors as follows. 


\begin{figure}
\begin{center}
\epsfbox{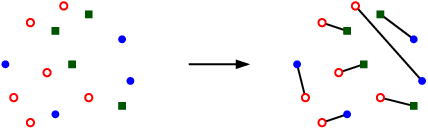}
\end{center}
\caption{A noncrossing geometric 
properly colored perfect matching.}
\label{fig_1_rgb}
\end{figure}



\begin{theorem}[Aichholzer et al.~\cite{ACFFHHHW10_edge}; Kano, Suzuki and Uno~\cite{KSU2014}]
\label{theorem_3_rcol}
 Let $r\ge 3$ and $n\ge 1$ be integers.
Let $X_1, X_2, \dots, X_r$ be $r$ disjoint point sets in the plane.
Assume that no three points of
$X_1 \cup X_2 \cup \dots \cup X_r$ lie on a line, 
 $\sum_{i=1}^r |X_i|=2n$, and
$|X_i| \le n$ for every $i\in [r]$.
Then there exists a noncrossing 
geometric perfect matching on $X_1\cup X_2 \cup \dots \cup X_r$ where every edge connects two points from different sets $X_i$ and $X_j$.
\end{theorem}

Aichholzer et al.~\cite{ACFFHHHW10_edge} proved Theorem~\ref{theorem_3_rcol} by the same metric argument as the in case of two colors. Kano, Suzuki and Uno~\cite{KSU2014} first proved Theorem~\ref{theorem_3_rcol} for three colors, by induction using a result on partitions of $3$-colored point sets on a line. Then they derived the case of four or more colors by merging the smallest color classes together.


Kano and Suzuki~\cite{KS2014} made the following conjecture generalizing Theorem~\ref{theorem_1_colorful_partition} and Theorem~\ref{theorem_3_rcol}.

\begin{conjecture}[Kano and Suzuki~\cite{KS2014}]
\label{conj_general_dim_col}
Let $r \ge d\ge 3$ and $n\ge 1$ be integers.
Let $X_1, X_2, \dots,\allowbreak X_r$ be $r$ disjoint point sets in $\mathbb{R}^d$.
Assume that no $d+1$ points of
$X_1 \cup X_2 \cup \dots \cup X_r$ lie in a hyperplane, 
 $\sum_{i=1}^r |X_i|=dn$, and
$|X_i| \le n$ for every $i\in [r]$.
Then there exist $n$ pairwise disjoint $(d-1)$-dimensional simplices, each of them having $d$ vertices in $d$ distinct sets $X_i$.
\end{conjecture}

Conjecture~\ref{conj_general_dim_col} holds when $r=d$ by Theorem~\ref{theorem_1_colorful_partition}
or $d=2$ by Theorem~\ref{theorem_3_rcol}.

In this paper we prove Conjecture~\ref{conj_general_dim_col} for every $d\ge 2$ and $r=d+1$. We restate it as the following theorem. 

\begin{theorem}
\label{theorem_main_for_points}
Let $d\ge 2$ and $n\ge 1$ be integers.
Let $X_1, X_2, \dots, X_{d+1}$ be $d+1$ disjoint point sets in $\mathbb{R}^d$.
Assume that no $d+1$ points of $X_1 \cup X_2 \cup \dots \cup X_{d+1}$ lie in a hyperplane, 
$\sum_{i=1}^{d+1} |X_i|=dn$, and
$|X_i| \le n$ for every $i\in [d+1]$.
Then there exist $n$ pairwise disjoint $(d-1)$-dimensional simplices, each of them having $d$ vertices in $d$ distinct sets $X_i$.
\end{theorem}

The proof of Theorem~\ref{theorem_main_for_points} (see Section~\ref{section_3}) provides yet another different proof of Theorem~\ref{theorem_3_rcol}.

Many related results and problems on colored point sets can be found in
a survey by Kaneko and Kano~\cite{KK2003}.

\subsection{Simultaneous partitions of measures}
\label{sec_ham_sandwich}

We denote by $S^n$ the $n$-dimensional unit sphere embedded in $\mathbb{R}^{n+1}$, that is, $S^n=\{\mathbf{x}\in\mathbb{R}^{n+1}; \lVert\mathbf{x}\rVert=1\}$.

The Borsuk--Ulam theorem plays an important role throughout this paper.

\begin{theorem}[The Borsuk--Ulam theorem{~\cite[Theorem 2.1.1]{Mato2003}}]
Let $f:S^n \to \mathbb{R}^n$ be a continuous mapping.
If $f(-\mathbf{u}) = -f(\mathbf{u})$ for all $\mathbf{u}\in S^n$,
then there exists a point $\mathbf{v}\in S^n$ such that
 $f\mathbf{(v}) = \mathbf{0}=(0, 0, \dots, 0)$.
\end{theorem}

Informally speaking, the ham sandwich theorem states that a sandwich made of bread, ham and cheese can be cut by a single plane, bisecting the mass of each of the three ingredients exactly in half. According to Beyer and Zardecki~\cite{BZ04_early_ham}, the ham sandwich theorem was conjectured by Steinhaus and appeared as Problem 123 in The Scottish Book~\cite{Scottish81}. Banach gave an elementary proof of the theorem using the Borsuk--Ulam theorem for $S^2$. A more direct proof can be obtained from the Borsuk--Ulam theorem for $S^3$~\cite{Mato2003}. Stone and Tukey~\cite{ST42_generalized} generalized the ham sandwich theorem to $d$-dimensional sandwiches made of $d$ ingredients as follows.

\begin{theorem}[The ham sandwich theorem~\cite{ST42_generalized}, {\cite[Theorem 3.1.1]{Mato2003}}]
\label{theorem_ham_sandwich} 
Let $\mu_1,\mu_2,  \dots$, $\mu_d$ be $d$ absolutely continuous finite Borel measures on
$\mathbb{R}^d$. Then there exists a hyperplane $h$ such that
each open halfspace $H$ defined by $h$ satisfies
$\mu_i(H)=\mu_i(\mathbb{R}^d)/2$ for every $i\in [d]$.
\end{theorem}

Stone and Tukey~\cite{ST42_generalized} proved more general versions of the ham sandwich theorem, including a version for Carath\'eodory outer measures and more general cutting surfaces. Cox and McKelvey~\cite{CM84_general} and Hill~\cite{Hill88_vectors} generalized the ham sandwich theorem to general finite Borel measures, which include measures with finite support. For these more general measures, the condition $\mu_i(H)=\mu_i(\mathbb{R}^d)/2$ must be replaced by the inequality $\mu_i(H)\le\mu_i(\mathbb{R}^d)/2$. Breuer~\cite{Br10_uneven} gave sufficient conditions for the existence of more general splitting ratios. In particular, he showed that for absolutely continuous measures whose supports can be separated by hyperplanes, there is a hyperplane splitting the measures in any prescribed ratio. See Matou\v{s}ek's book~\cite{Mato2003} for more generalizations of the ham sandwich theorem and other related partitioning results.

To prove Theorem~\ref{theorem_main_for_points}, we follow the approach by Akiyama and Alon~\cite{AA1989}. To this end, we need to generalize the ham sandwich theorem to $d+1$ measures on $\mathbb{R}^d$. Clearly, it is not always possible to bisect all $d+1$ measures by a single hyperplane, for example, if each measure is concentrated on a small ball around a different vertex of a regular simplex.

Let $r\ge d$ and let $\mu_1, \mu_2, \dots, \mu_{r}$ be finite Borel measures on $\mathbb{R}^d$. We say that $\mu_1,\allowbreak \mu_2, \dots,\allowbreak \mu_{r}$ are \emph{balanced} in a subset $X\subseteq \mathbb{R}^d$ if for every $i\in [r]$, we have
\[
\mu_i(X) \le \frac{1}{d}\cdot \sum_{j=1}^{r} \mu_j(X).
\]

\begin{theorem}[The hamburger theorem]
\label{theorem_hamburger}
Let $d\ge 2$ be an integer.
Let $\mu_1, \mu_2, \dots, \mu_{d+1}$ be absolutely continuous finite Borel measures on
$\mathbb{R}^d$. Let $\omega_i=\mu_i(\mathbb{R}^d)$ for $i\in [d+1]$ and $\omega=\min\{\omega_i; i\in [d+1]\}$.
Assume that $\sum_{j=1}^{d+1} \omega_j=1$ and that $\mu_1, \mu_2, \dots, \mu_{d+1}$ are balanced in $\mathbb{R}^d$.
Then there exists a hyperplane $h$ 
such that for each open halfspace $H$ defined by $h$, the measures $\mu_1, \mu_2, \dots, \mu_{d+1}$ are balanced in $H$ and $\sum_{j=1}^{d+1} \mu_j(H) \ge \min\{1/2, 1-d\omega\} \ge 1/(d+1)$.

Moreover, setting $t=\min\{1/(2d), 1/d - \omega\}$ and assuming that $\omega_{d+1}=\omega$, the vector $(\mu_1(H),\allowbreak \mu_2(H),\dots,\allowbreak \mu_{d+1}(H))$ is a convex combination of the vectors $(t,t,\dots,t,0)$ and $(\omega_1-t, \omega_2-t, \dots, \omega_d-t,\omega_{d+1})$.
\end{theorem}

Our choice of the name for Theorem~\ref{theorem_hamburger} is motivated by the fact that compared to a typical ham sandwich, a typical hamburger consists of more ingredients, such as bread, beef, bacon, and salad.

Note that the lower bound $\min\{1/2, 1-d\omega\}$ on the total measure of the two halfspaces is tight: consider, for example, $d+1$ measures such that each of them is concentrated on a small ball centered at a different vertex of the unit $d$-dimensional simplex. 


\section{Proof of the hamburger theorem}

In this section we prove Theorem~\ref{theorem_hamburger}. We follow the proof of the ham sandwich theorem for measures from~\cite{Mato2003}.


The set of open half-spaces in $\mathbb{R}^d$, together with the empty set and the whole space $\mathbb{R}^d$, has a natural topology of the sphere $S^d$. We use the following parametrization.

Let $\mathbf{u}=(u_0,\allowbreak u_1, \dots,\allowbreak u_d)$ be a point from the sphere $S^d$, that is, 
$u_0^2+u_1^2+ \dots +u_d^2=1$. If $|u_0|<1$, then at least one of the coordinates $u_1, u_2, \dots,\allowbreak u_d$ is nonzero, and we define two halfspaces as follows:
\begin{align*}
&  H^-(\mathbf{u})=\{ (x_1, x_2, \dots, x_d)\in \mathbb{R}^d; u_1x_1 + u_2x_2+ \dots +u_d x_d < u_0 \}, \\
&  H^+(\mathbf{u})=\{ (x_1, x_2, \dots, x_d)\in \mathbb{R}^d; u_1x_1 + u_2x_2+ \dots +u_d x_d > u_0 \}.
\end{align*}
We also define a hyperplane $h(\mathbf{u})$ as the common boundary of $H^-(\mathbf{u})$ and $H^+(\mathbf{u})$.
For the two remaining points $(1,0,\dots,0)$ and $(-1,0,\dots,0)$, we set

\begin{align*}
H^-(1,0,0, \dots, 0) &= \mathbb{R}^d, & H^+(1,0,0, \dots, 0) &= \emptyset,\\
H^-(-1,0,0, \dots, 0)&=\emptyset, & H^+(-1,0,0, \dots, 0) &=\mathbb{R}^d.
\end{align*}

Note that antipodal points on $S^d$ correspond to complementary half-spaces; that is, $H^-(\mathbf{u})=H^+(\mathbf{-u})$ for every $\mathbf{u}\in S^d$.

We define a function $f=(f_1, f_2,\dots, f_{d+1}):S^d \rightarrow \mathbb{R}^{d+1}$ by
\[
f_i(\mathbf{u})=\mu_i(H^-(\mathbf{u})).
\]
Since the measures $\mu_i$ are absolutely continuous, $\mu_i(h(\mathbf{u}))=0$ for every hyperplane $h(\mathbf{u})$. This implies that $f$ is continuous~\cite{Mato2003}. 

The image of $f$ lies in the box $B=\prod_{i=1}^{d+1} [0,\omega_i]$. Moreover, $f$ maps antipodal points of the sphere to points symmetric about the center of $B$. 
The \emph{target polytope} is the subset of points $(y_1, y_2, \dots, y_{d+1})$ of $B$ satisfying the inequalities
\begin{equation}
\label{equation_balanced}
y_i\le \frac{1}{d}\cdot\sum_{j=1}^{d+1}y_j \hskip 5mm \text{ and } \hskip 5mm  \omega_i-y_i\le \frac{1}{d}\cdot\sum_{j=1}^{d+1} (\omega_j-y_j). 
\end{equation}
See Figure~\ref{fig_2_target_polytope}, left. The subset of the target polytope satisfying the inequalities 
\begin{equation}
\label{equation_truncated}
\min\{1/2, 1-d\omega\}\le y_1+y_2+\cdots +y_{d+1}\le 1-\min\{1/2, 1-d\omega\}
\end{equation}
is called the \emph{truncated target polytope}; see Figure~\ref{fig_2_target_polytope}, right.
Our goal is to show that the image of $f$ intersects the truncated target polytope.

\begin{figure}
\begin{center}
\epsfbox{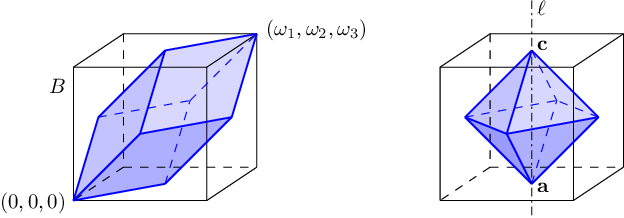}
\end{center}
\caption{Left: the target polytope inside $B$ for $d=2$ and $\omega_1=\omega_2=\omega_3=1/3$. Right: the truncated target polytope corresponding to hyperplanes cutting at least $1/3$ of the total measure on both sides. The segment $\mathbf{a}\mathbf{c}$, which is the intersection of the line $\ell$ with $B$, is contained in the truncated target polytope.}
\label{fig_2_target_polytope}
\end{figure}

We first show a proof using the notion of a \emph{degree} of a map between spheres. Then we modify it so that it uses only the Borsuk--Ulam theorem.

Let $\mathbf{b}=(\omega_1/2,\allowbreak \omega_2/2, \dots,\allowbreak \omega_{d+1}/2)$ be the center of $B$. The map $g=f-\mathbf{b}$ is antipodal, that is, $g(-\mathbf{u})=-g(\mathbf{u})$ for every $\mathbf{u}\in S^d$. Clearly, $\mathbf{b}$ satisfies both (\ref{equation_balanced}) and (\ref{equation_truncated}) and thus it belongs to the truncated target polytope. Hence, if $\mathbf{0}$ is in the image of $g$, then any hyperplane $h(\mathbf{u})$ such that $g(\mathbf{u})=\mathbf{0}$ satisfies the theorem. 

For the rest of the proof we may assume that $\mathbf{0}$ is not in the image of $g$. Then we can define an antipodal map $\tilde{g}:S^d\rightarrow S^d$ as 
\[
\tilde{g}(\mathbf{u})=\frac{g(\mathbf{u})}{\lVert g(\mathbf{u})\rVert}.
\]
Using the fact that every antipodal map from $S^d$ to itself has odd degree~\cite[Proposition 2B.6.]{Hatcher02_AT}, we conclude that $\tilde{g}$ is surjective. Hence, the image of $g$ intersects every line passing through the origin, equivalently, the image of $f$ intersects every line passing through $\mathbf{b}$. Therefore, it is sufficient to find a line $\ell$ through $\mathbf{b}$ such that $\ell\cap B$
belongs to the truncated target polytope.

Without loss of generality, we may assume that $\omega_1 \ge \omega_2 \ge \dots \ge \omega_{d+1} = \omega$. Let 
\[
t=\min\left\{\frac{1}{2d}, \frac{1}{d}-\omega_{d+1}\right\}.
\]
We define $\ell$ as the line containing the points 
\begin{align*}
\mathbf{a}&=(t,t,\dots,t,0) \ \text{ and } \\
\mathbf{c}&=(\omega_1-t, \omega_2-t, \dots, \omega_d-t,\omega_{d+1}). 
\end{align*}
Since the measures $\mu_1, \mu_2, \dots, \mu_{d+1}$ are balanced in $\mathbb{R}^d$, we have $\omega_i\le 1/d$ for all $i\in [d+1]$. Thus $\omega_d+\omega_{d+1}=1-(\omega_1 + \omega_2 + \cdots + \omega_{d-1})\ge 1/d$ and consequently $\omega_d \ge 1/(2d)\ge t$. This implies that both points $\mathbf{a}$ and $\mathbf{c}$ lie in $B$. Moreover, $\mathbf{a}$ and $\mathbf{c}$ lie on the opposite facets of $B$ and they are symmetric around the center $\mathbf{b}$. The points $\mathbf{a}$ and $\mathbf{c}$ both satisfy (\ref{equation_balanced}) since $(\omega_1-t)+(\omega_2-t)+ \cdots + (\omega_d-t)+\omega_{d+1} = 1-dt \ge d\omega_{d+1}$. They also both satisfy (\ref{equation_truncated}) since $dt\le 1-dt$. Therefore, the segment $\mathbf{a}\mathbf{c}$ is the intersection of the line $\ell$ with $B$ and it is contained in the truncated target polytope.

  Now we show how to replace the degree argument with the application of the Borsuk--Ulam theorem. Let $\pi_{\ell}$ be a projection of $\mathbb{R}^{d+1}$ in the direction of the line $\ell$ to a $d$-dimensional subspace orthogonal to $\ell$, which we identify with $\mathbb{R}^d$. Define a map $g':S^d\rightarrow \mathbb{R}^d$ by $g'(\mathbf{u})= \pi_{\ell}(g(\mathbf{u}))$. The map $g'$ is antipodal, and so by the Borsuk--Ulam theorem, there exists $\mathbf{u}\in S^d$ such that $g'(\mathbf{u})=\mathbf{0}$, which means that $f(\mathbf{u})\in\ell$. This concludes the proof.

\section{A discrete version of the hamburger theorem}
\label{section_3}

Theorem~\ref{theorem_main_for_points} follows by induction from the following discrete analogue of the hamburger theorem. 

We say that point sets $X_1, X_2, \dots, X_{r}$ are \emph{balanced} in a subset $S\subseteq \mathbb{R}^d$ if for every $i\in[r]$, we have
\[
\lvert S\cap X_i\rvert \le \frac{1}{d}\cdot \sum_{j=1}^{r} \lvert S\cap X_j\rvert.
\]

\begin{theorem}
\label{theorem_discrete}
Let $d\ge 2$ and $n\ge 2$ be integers.
Let $X_1, X_2, \dots, X_{d+1}\subset \mathbb{R}^d$ be $d+1$ disjoint point sets balanced in $\mathbb{R}^d$.
Assume that no $d+1$ points of $X_1 \cup X_2 \cup \dots \cup X_{d+1}$ lie in a hyperplane and that $\sum_{i=1}^{d+1} |X_i|=dn$.
Then there exists a hyperplane $h$ disjoint with each $X_i$ such that 
for each open halfspace $H$ determined by $h$, the sets $X_1, X_2, \dots, X_{d+1}$ are balanced in $H$ and $\sum_{i=1}^{d+1} |H\cap X_i|$ is a positive integer multiple of $d$.
\end{theorem}

While preparing the final version of this paper, we have learned that Biniaz, Maheshwari, Nandy and Smid~\cite{BMNS15_an_optimal} independently proved the plane version of Theorem~\ref{theorem_discrete}, and they gave a linear-time algorithm for computing the cutting line.

\subsection{Proof of Theorem~\ref{theorem_discrete}}

\paragraph{Approximating point sets by measures.}
Let $X=\bigcup_{i=1}^{d+1} X_i$. 
Replace each point $\mathbf{p}\in X$ by an open ball $B(\mathbf{p})$ of a sufficiently small radius $\delta>0$ centered in $\mathbf{p}$, so that no hyperplane intersects more than $d$ of these balls. We will apply the hamburger theorem for suitably defined measures supported by the balls $B(\mathbf{p})$. Rather than taking the same measure for each of the balls, we use a variation of the trick used by Elton and Hill~\cite{EH11_stronger_conclusion}. For each $\mathbf{p}\in X$ and $k\ge 1$, we choose a number $\beta_k(\mathbf{p})\in (1-1/k,1+1/k)$ so that the following two conditions are satisfied.
\begin{enumerate}
\item[I)] For every $i\in [d+1]$, we have
%
$\sum_{\mathbf{p}\in X_i} \beta_k(\mathbf{p}) = \lvert X_i\rvert$. 
%
\item[II)] Let $\omega_i=\lvert X_i \rvert /\lvert X\rvert$ for $i\in [d+1]$, suppose that $\omega_{d+1}=\min\{\omega_i; i\in [d+1]\}$ and let $t=\min\{1/(2d), 1/d - \omega_{d+1}\}$.
For every $i,j\in [d+1]$, $i\neq j$, and for every pair of proper nonempty subsets $Y\subset X_i$ and $Z\subset X_j$, there is no vector $(\gamma_1,\gamma_2,\dots,\gamma_{d+1})\in\mathbb{R}^{d+1}$ that is a convex combination of the vectors $(t,t,\dots,t,0)$ and $(\omega_1-t, \omega_2-t, \dots, \omega_d-t,\omega_{d+1})$ and satisfies $\gamma_i=\sum_{\mathbf{p}\in Y} \beta_k(\mathbf{p})$ and $\gamma_j=\sum_{\mathbf{p}\in Z} \beta_k(\mathbf{p})$.
\end{enumerate}

Now let $k\ge 1$ be a fixed integer. For each $i\in [d+1]$, let $\mu_{i,k}$ be the measure supported by the closure of $\bigcup_{\mathbf{p}\in X_i} B(\mathbf{p})$ such that it is uniform (that is, equal to a multiple of the Lebesgue measure) on each of the balls $B(\mathbf{p})$ and 
$\mu_{i,k}(B(\mathbf{p}))=\beta_k(\mathbf{p})$. 

\paragraph{Applying the hamburger theorem and a final rounding step.}
By condition I), the measures $\mu_{1,k}, \mu_{2,k}, \dots, \mu_{d+1,k}$ are balanced in $\mathbb{R}^d$. 
We may thus apply the hamburger theorem to the normalized collection of measures 
$\mu_{i,k}/\lvert X \rvert$.
Let $h_k$ be the resulting hyperplane. We distinguish two cases.

1) We have $\mu_{i,k}(H)=0$ for some $i\in[d+1]$ and some halfspace $H$ determined by $h_k$. Since the measures $\mu_{1,k}, \mu_{2,k}, \dots, \mu_{d+1,k}$ are balanced in $H$, there is an $\alpha>0$ such that $\mu_{j,k}(H)=\alpha$ for every $j\in[d+1]\setminus \{i\}$.

2) The hyperplane $h_k$ splits each measure $\mu_{i,k}$ in a nontrivial way. By condition II), $h_k$ intersects the support of exactly $d$ of the measures $\mu_{i,k}$.

For each $i\in [d+1]$, let $\mu_i$ be the limit of the measures $\mu_{i,k}$ when $k$ grows to infinity; that is, $\mu_i$ is uniform on every ball $B(\mathbf{p})$ such that $\mathbf{p}\in X_i$ and $\mu_i(B(\mathbf{p}))=1$. Since the supports of all the measures $\mu_{i,k}$ are uniformly bounded, there is an increasing sequence $\{k_m; m=1,2,\dots\}$ such that the sequence of hyperplanes $h_{k_m}$ has a limit $h'$. More precisely, if $h_{k_m}=\{\mathbf{x}\in \mathbb{R}^d; \mathbf{x} \cdot \mathbf{v}_m=c_m\}$ where $\mathbf{v}_m \in S^{d-1}$, then $h'=\{\mathbf{x}\in \mathbb{R}^d; \mathbf{x} \cdot \mathbf{v}=c\}$ where $\mathbf{v} = \lim_{m\rightarrow \infty} \mathbf{v}_m$ and $c = \lim_{m\rightarrow \infty} c_m$. By the absolute continuity of the measures, the measures $\mu_1, \mu_2, \dots, \mu_{d+1}$ are balanced in each of the two halfspaces determined by $h'$, and the total measure $\sum_{j=1}^{d+1}\mu_j$ of each of the two halfspaces is at least $n/(d+1)$. One of the cases 1) or 2) occurred for infinitely many hyperplanes $h_{k_m}$, $m\ge 1$. We distinguish the two possibilities.

Suppose that case 1) occurred infinitely many times. For each $m\ge 1$, let $H^+_{k_m}=\{\mathbf{x}\in \mathbb{R}^d; \mathbf{x} \cdot \mathbf{v}_m>c_m\}$ and $H^-_{k_m}=\{\mathbf{x}\in \mathbb{R}^d; \mathbf{x} \cdot \mathbf{v}_m<c_m\}$ be the two open halfspaces determined by $h_{k_m}$.
There is an $i\in [d+1]$ such that $\mu_{i,k_m}(H^+_{k_m})=0$ occurred for infinitely many $m\ge 1$ or $\mu_{i,k_m}(H^-_{k_m})=0$ occurred for infinitely many $m\ge 1$. By the absolute continuity of the measures, there is an $\alpha\ge n/(d+1)>0$ such that for one of the halfspaces $H$ determined by $h'$ and for every $j\in[d+1]\setminus \{i\}$, we have $\mu_{j}(H)=\alpha$. 

If $h'$ is disjoint from all the balls $B(\mathbf{p})$, $\mathbf{p}\in X$, then this hyperplane satisfies the conditions of the theorem. Otherwise, $h'$ intersects one ball from the support of each of the measures $\mu_j$, $j\in[d+1]\setminus \{i\}$. Let $h$ be the translation of $h'$ that touches each of these $d$ balls and such that $h$ is in the complement of $H$. Let $\tilde{H}$ be the open halfspace determined by $h$ containing $H$, and let $\tilde{H}'$ be the open halfspace opposite to $\tilde{H}$. Then for every $j\in[d+1]\setminus \{i\}$, we have $\mu_{j}(\tilde{H})=\lceil\alpha\rceil$. In particular, the sets $X_1, X_2, \dots, X_{d+1}$ are balanced in $\tilde{H}$ and $0<|\tilde{H} \cap X|=d\lceil\alpha\rceil<dn$. It remains to show that $X_1, X_2, \dots, X_{d+1}$ are also balanced in $\tilde{H}'$. Let $H'$ be the open halfspace opposite to $H$. Since the measures $\mu_1, \mu_2, \dots, \mu_{d+1}$ are balanced in $H'$, for each $k\in[d+1]\setminus \{i\}$ we have
\begin{align*}
\mu_k(\tilde{H}') &= (\alpha- \lceil\alpha\rceil)+ \mu_k(H')\le (\alpha-\lceil\alpha\rceil) + \frac{1}{d}\cdot\sum_{j=1}^{d+1}\mu_j(H')=\frac{1}{d}\cdot \sum_{j=1}^{d+1}\mu_j(\tilde{H}'),
\end{align*}
and for the measure $\mu_i$ we have
\begin{align*}
\mu_i(\tilde{H}') &= \mu_i(H')\le \frac{1}{d}\cdot\sum_{j=1}^{d+1}\mu_j(H') = \frac{1}{d}\cdot\left(\sum_{j=1}^{d+1}\mu_j(\tilde{H}')\right) + (\lceil\alpha\rceil-\alpha)\\
&=(n-\lceil\alpha\rceil) + (\lceil\alpha\rceil-\alpha) = n-\alpha.
\end{align*}
Since $\mu_i(\tilde{H}')=|X_i|$, we can replace the upper bound by the nearest integer, and thus we have 
\[
\mu_i(\tilde{H}') \le n-\lceil\alpha\rceil = \frac{1}{d}\cdot\sum_{j=1}^{d+1}\mu_j(\tilde{H}').
\]


Finally, suppose that case 2) occurred infinitely many times. 
There is an $i\in [d+1]$ such that for infinitely many $m\ge 1$, the hyperplane $h_{k_m}$ intersects the support of each measure $\mu_{j,k_m}$, $j\in [d+1]\setminus \{i\}$. In particular, for each $j\in [d+1]\setminus \{i\}$, there is a point $\mathbf{p}_j\in X_j$ such that for infinitely many $m\ge 1$, the hyperplane $h_{k_m}$ intersects each of the balls $B(\mathbf{p}_j)$. It follows that $h'$ intersects or touches each of the balls $B(\mathbf{p}_j)$. Moreover, each of the halfspaces determined by $h'$ contains at least one point of $X_i$. 


We now rotate the hyperplane $h'$ slightly, to a hyperplane $h$ that touches each of the balls $B(\mathbf{p}_j)$ with $j\in [d+1]\setminus \{i\}$, so that the point sets $X_1, X_2, \dots, X_{d+1}$ are balanced in each halfspace determined by $h$ and the number of points of $X$ in each halfspace is divisible by $d$. Each of the two halfspaces will still contain a positive number of points from $X_i$, so the resulting partition of $X$ will be nontrivial. Essentially, for each point $\mathbf{p}_j$,  $j\in [d+1]\setminus \{i\}$, we can decide independently on which side of $h$ it will end, thus we are choosing one of $2^d$ possible rotations. Let $Y=\{\mathbf{p}_j; j\in [d+1]\setminus \{i\}\}$ be the set of these $d$ ``movable'' points.

Let $H'_A$ and $H'_B$ be the two open halfspaces determined by $h'$. We may consider $h'$ to be horizontal, $H'_A$ to be the halfspace above $h'$ and $H'_B$ the halfspace below $h'$. We now count the points of $X\setminus Y$, which will remain above both $h'$ and $h$ or below both $h'$ and $h$, no matter which rotation of $h'$ we choose.
For each $j\in [d+1]\setminus \{i\}$, let $a'_j=|(X_j\setminus\{\mathbf{p}_j\})\cap H'_A|$ and $b'_j=|(X_j\setminus\{\mathbf{p}_j\})\cap H'_B|$. Let $a'_i=|X_i\cap H'_A|$ and $b'_i=|X_i\cap H'_B|$. Clearly, we have $|X_j|=a'_j+b'_j+1$ if $j\neq i$ and $|X_i|=a'_i+b'_i$. Finally, let $a'=|(X\setminus Y)\cap H'_A|$ and $b'=|(X\setminus Y)\cap H'_B|$. Clearly, we have $a'=\sum_{j=1}^{d+1} a'_j$, $b'=\sum_{j=1}^{d+1} b'_j$ and $a'+b'=d(n-1)$.

For each $j\in [d+1]\setminus \{i\}$, let $\alpha_j=\mu_j(B(\mathbf{p}_j)\cap H'_A)$, and let $\alpha=\sum_{j\in[d+1]\setminus\{i\}}\alpha_j$. Clearly, we have $0\le \alpha_j\le 1$, $0\le \alpha \le d$, $\mu_j(H'_A)=a'_j+\alpha_j$, $\mu_j(H'_B)=b'_j+1-\alpha_j$, the total measure  $\sum_{j=1}^{d+1}\mu_j$ of  $H'_A$ is equal to $a'+\alpha$, and the total measure of $H'_B$ is equal to $b'+d-\alpha$.

Let $m_A=\max\{a'_j; j\in [d+1]\}$ and $m_B=\max\{b'_j; j\in [d+1]\}$.
Since the measures $\mu_1, \mu_2, \dots, \mu_{d+1}$ are balanced in $H'_A$ and also in $H'_B$, we have $dm_A\le a'+\alpha$ and $dm_B\le b'+ d-\alpha$, which implies that $m_A+m_B\le n$.

To get a positive integer multiple of the points of $X$ in each halfspace determined by $h$, the numbers of points of $Y$ that have to be moved above and below $h$ are uniquely determined, except in the case when $a'$ and $b'$ are both divisible by $d$, in which case either all points of $Y$ have to be moved above $h$ or all below $h$. We consider this case separately.

Assume that $a'$ and $b'$ are both divisible by $d$; that is, $a'=dn'_A$ and $b'=dn'_B$ for some positive integers $n'_A,n'_B$. If $\alpha=0$ or $\alpha=d$, the hyperplane $h'$ already satisfies the conditions of the theorem and we can put $h=h'$. If $\alpha \in (0,d)$, then $m_A\le n'_A$ and $m_B \le n'_B$. In this case, we can move all the points of $Y$ above $h$ (alternatively, all below $h$). Then there will be exactly $d(n'_A+1)$ points of $X$ above $h$, $dn'_B$ points of $X$ below $h$, and each $X_j$ will have at most $n'_A+1$ points above $h$ and at most $n'_B$ points below $h$. Thus $h$ will satisfy the conditions of the theorem.

For the rest of the proof, assume that $a'$ and $b'$ are not divisible by $d$, and let $r=b' \text{ mod } d$. Then exactly $d-r$ points of $Y$ have to be moved below $h$ and the remaining $r$ points above $h$.

If $m_A+m_B= n$, then necessarily $dm_A= a'+\alpha$, $dm_B= b'+ d-\alpha$ and $r=\alpha$. We need to select a subset of $r$ points of $Y$ that will be moved above $h$, so that each $X_j$ has still at most $m_A$ points above $h$, and similarly with points that are moved below $h$. 
For $j\in [d+1]$, call a point set $X_j$ \emph{saturated from above} if $a'_j=m_A$, and \emph{saturated from below} if $b'_j=m_B$. For $j\neq i$, the set $X_j$ cannot be saturated both from above and below, since $a'_j+b'_j=|X_j|-1\le n-1= m_A+m_B-1$. If $X_j$ is not saturated from above, we can move $\mathbf{p}_j$ above $h$, and if $X_j$ is not saturated from below, we can move $\mathbf{p}_j$ below $h$. It remains to verify that among the sets $X_j, j\in [d+1]\setminus \{i\}$, at most $r=\alpha$ of them are saturated from below and at most $d-r=d-\alpha$ of them saturated from above. This follows from the facts that if $X_j$ is saturated from above and $j\neq i$, then $\alpha_j=0$; and similarly, if $X_j$ is saturated from below and $j\neq i$, then $1-\alpha_j=0$.

Now assume that $m_A+m_B\le n-1$. In this case we have $dm_A\le a'$ or $dm_B\le b'$. By symmetry, we may assume that $dm_B\le b'$, which implies a stronger inequality $d(m_B+1)\le b'-r+d$. This means that we can move an arbitrary subset of $d-r$ points of $Y$ below $h$ to keep the point sets $X_1, X_2, \dots, X_{d+1}$ balanced in the halfspace below $h$. Since $dm_A\le a'+ \alpha$, we have $da'_i \le a'+r$. For $j\in [d+1]\setminus \{i\}$, call a point set $X_j$ \emph{saturated from above} if $da'_j=a'+r$. Let $s$ be the number of the sets $X_j$, $j\neq i$, that are saturated from above. We need to verify that $s\le d-r$. If $X_j$ is saturated from above then $a'_j=m_A$ and thus $r\le \alpha$. Furthermore, for such $X_j$ we have $d(a'_j+\alpha_j)\le a'+ \alpha$, which implies $\alpha_j\le (\alpha-r)/d$. Hence,  
\[
\alpha=\sum_{j\in[d+1]\setminus\{i\}}\alpha_j \le s(\alpha-r)/d + d-s = d-s(d+r-\alpha)/d. 
\]
Using the inequalities $0<r\le\alpha\le d$, this further implies
\[
s\le\frac{d^2-d\alpha}{d+r-\alpha} \le \frac{d^2-d\alpha+r(\alpha-r)}{d+r-\alpha} = \frac{(d+r-\alpha)(d-r)}{d+r-\alpha} = d-r.
\]
%
Therefore we can move $r$ points from $Y$ above $h$ to keep the point sets $X_1, X_2, \dots, X_{d+1}$ balanced in the halfspace above $h$.

\section{Concluding remarks}

We were not able to generalize the hamburger theorem for $d\ge 3$ and $d+2$ or more measures on $\mathbb{R}^d$, even if instead of the condition that the hyperplane cuts at least $1/(d+1)$ of the total measure on each side,  we require only that the partition is nontrivial.

\begin{problem}
\label{problem_measures}
Let $d\ge 3$ and $r\ge d+2$ be integers.
Let $\mu_1, \mu_2, \dots, \mu_{r}$ be absolutely continuous positive finite Borel measures on
$\mathbb{R}^d$ that are balanced in $\mathbb{R}^d$. Does there exists a hyperplane $h$ such that for each open halfspace $H$ defined by $h$, the total measure $\sum_{j=1}^{r} \mu_j(H)$ is positive and the measures $\mu_1, \mu_2, \dots, \mu_{r}$ are balanced in $H$?
\end{problem}

It is easy to see that for a given $d$, it would be sufficient to prove Problem~\ref{problem_measures} for $r\le 2d-1$. Indeed, suppose that $r\ge 2d$ and that $\mu_1, \mu_2, \dots, \mu_{r}$ are measures balanced in $\mathbb{R}^d$. If $\mu_1(\mathbb{R}^d)\ge \mu_2(\mathbb{R}^d) \ge \dots \ge \mu_r(\mathbb{R}^d)$, then after replacing $\mu_{r-1}$ and $\mu_r$ by a single measure $\mu'_{r-1}=\mu_{r-1}+\mu_r$, the resulting set of $r-1$ measures is still balanced in $\mathbb{R}^d$.

For the case of five measures on $\mathbb{R}^3$, our approach from the proof of the hamburger theorem fails for the following reason. If $\mu_i(\mathbb{R}^3)=1/5$ for each $i\in[5]$, then the target polytope, now in $\mathbb{R}^5$, intersected with the boundary of the box $B$, does not contain a closed curve symmetric with respect to the center of $B$. If such a curve existed, we could apply a generalization of the Borsuk--Ulam theorem saying that if $f:S^k\rightarrow S^{k+l}$ and $g:S^l\rightarrow S^{k+l}$ are antipodal maps, then their images intersect~\cite[Exercise $3.^*/116$]{Mato2003}.

\paragraph{Update.}
After finishing this paper, D\"om\"ot\"or P\'alv\"olgyi gave a negative solution to Problem~\ref{problem_measures} for $d\ge 4$ by the following construction, which we include here with his permission. Take a $d$-dimensional simplex with vertices $v_1, v_2, \dots, v_{d+1}$ centered at the origin, and for each $i\in[d+1]$, let $\mu_i$ be a probabilistic measure concentrated on a small ball centered at $v_i$. Let $\mu_{d+2}$ be a probabilistic measure concentrated on $d+1$ small balls centered at the vertices $2v_1, 2v_2, \dots, 2v_{d+1}$ of a larger simplex, such that for each $i\in[d+1]$, the ball centered at $2v_i$ has measure $1/(d+1)$. 

For the case $d=3$ and $r=5$, we only have a construction showing that a constant positive fraction of the total measure in each halfspace cannot be guaranteed. Consider a simplex with vertices $v_1,v_2,v_3,v_4$, and for each $i\in[4]$, let $\mu_i$ be a measure concentrated on a small ball $B_i$ centered at $v_i$ such that $\mu_i(B_i)=1+\varepsilon$. Let $\mu_5$ be a measure concentrated on a small ball $B_5$ centered at the origin such that $\mu_5(B_5)=2-4\varepsilon$. For every hyperplane $h$ satisfying the conditions of Problem~\ref{problem_measures}, $h$ must be disjoint with $B_5$ and so at least one halfspace determined by $h$ has total measure at most $9\varepsilon$.

\section*{Acknowledgements}
We thank J\'anos Pach and the members of the DCG group at EPFL for helpful discussions. We also thank Seunghun Lee for finding an error in the previous proof of Theorem~\ref{theorem_discrete}.

\end{document}